\documentclass[11pt,reqno]{amsart}
\usepackage[utf8]{inputenc}
\usepackage{enumitem}
\usepackage{hyperref}
\hypersetup{%
  pdftitle   = {Sasakian manifolds and M-theory},
  pdfkeywords = {Sasakian geometry, Killing spinors, supersymmetry, M-theory},
  pdfauthor  = {José Figueroa-O'Farrill, Andrea Santi},
  pdfcreator = {\LaTeX\ with package \flqq hyperref\frqq}
}
\numberwithin{equation}{section}
\theoremstyle{plain}
\newtheorem{theo}{Theorem}[section]

\theoremstyle{plain}
\newtheorem{theorem}[theo]{Theorem}
\newtheorem{proposition}[theo]{Proposition}
\theoremstyle{definition}
\newtheorem{definition}[theo]{Definition}
\newcommand{\bC}{\mathbb{C}}
\newcommand{\bR}{\mathbb{R}}
\newcommand{\bS}{\mathbb{S}}
\newcommand{\so}{\mathfrak{so}}
\newcommand\SO{\mathrm{SO}}
\newcommand\SU{\mathrm{SU}}
\newcommand\Spin{\mathrm{Spin}}
\newcommand{\cF}{\mathcal{F}}
\newcommand{\cH}{\mathcal{H}}
\newcommand{\cL}{\mathcal{L}}
\newcommand{\cV}{\mathcal{V}}
\DeclareMathOperator\Tr{Tr}
\DeclareMathOperator\Ad{Ad}
\DeclareMathOperator\vol{vol}
\DeclareMathOperator\Id{Id}
\DeclareMathOperator\Ric{Ric}
\newcommand\Ker{\operatorname{Ker}}
\newcommand{\wt}{\widetilde}
\newcommand{\wh}{\widehat}

\begin{document}
\title{Sasakian manifolds and M-theory}
\author[Figueroa-O'Farrill, Santi]{José Figueroa-O'Farrill, Andrea Santi}
\address{Maxwell Institute and School of Mathematics, The University
  of Edinburgh, James Clerk Maxwell Building, Peter Guthrie Tait Road,
  Edinburgh EH9 3FD, Scotland, UK}
\email{j.m.figueroa@ed.ac.uk, asanti.math@gmail.com}
\thanks{EMPG-15-18}
\keywords{Sasakian manifolds, Calabi-Yau manifolds, M-theory}
\subjclass[2000]{53C25, 53C27, 83E50}
\begin{abstract}
  We extend the link between Einstein Sasakian manifolds and Killing
  spinors to a class of $\eta$-Einstein Sasakian manifolds, both in
  Riemannian and Lorentzian settings, characterising them in terms of
  generalised Killing spinors.  We propose a definition of
  supersymmetric M-theory backgrounds on such a geometry and find a new
  class of such backgrounds, extending previous work of Haupt, Lukas
  and Stelle.
\end{abstract}
\maketitle
\tableofcontents

\section{Introduction}
\label{sec:introduction}

Sasakian manifolds (see, e.g., \cite{MR2382957}) continue to play an
important rôle in mathematical physics, ever since the emergence,
almost two decades ago, of the conjectural gauge/gravity
correspondence \cite{Malda}. Klebanov and Witten \cite{KlebanovWitten}
(following from earlier work of Kehagias \cite{Kehagias}) conjectured
that the gravity dual of a certain 4-dimensional $N=1$ superconformal
field theory was given by type IIB superstring theory on the product of
5-dimensional anti-de~Sitter spacetime and a homogeneous
five-dimensional Sasaki--Einstein manifold called $T^{1,1}$.  They
interpreted this ten-dimensional Lorentzian manifold as the
near-horizon geometry of a stack of D3-branes sitting at the
singularity of the conifold.

This interpretation was further explored and extended in
\cite{JMFSUSY98,AFHS,MorrisonPlesser}, setting up a correspondence
between superconformal field theories with less than maximal
supersymmetry and near-horizon geometries of supersymmetric brane
configurations, where the branes are located at a conical singularity
in a Riemannian manifold of special holonomy. The near-horizon
geometry of such branes is then metrically a product of an
anti-de~Sitter spacetime with the link of the cone, which is an
Einstein manifold (or, more generally, an orbifold) admitting real
Killing spinors. In particular, conical singularities of Calabi--Yau 
manifolds have links which are Sasakian manifolds. Indeed, one of the
equivalent characterisations of a Sasakian manifold is one whose
metric cone is Kähler and if, in addition, the Sasakian manifold is
Einstein with positive scalar curvature, then the cone is Ricci-flat
and hence Calabi--Yau.  This is one instance of Bär's cone
construction \cite{Baer}, which states that the metric cone of an
Einstein manifold admitting real Killing spinors is either flat or
irreducible and admits parallel spinors.

Although the gauge/gravity correspondence exists between
superconformal field theories and string/M-theory, it is the 't~Hooft
(or large $N$) limit that has been studied the most, since that
limit corresponds to the supergravity limit of the string or M-theory.
It is believed that supersymmetric supergravity backgrounds which are
dual to the large $N$ limit of a superconformal field theory can be
corrected (in a way analogous to the $1/N$ corrections of the field
theory) to yield exact string/M-theory backgrounds, but in the case of
M-theory this is hindered by the lack of a good working definition of
the notion of a ``supersymmetric M-theory background''.

The supergravity limit of M-theory is eleven-dimensional supergravity,
the unique eleven-dimensional supergravity theory with 32 supercharges
predicted by Nahm \cite{Nahm} and constructed by Cremmer, Julia and
Scherk \cite{CJS}. The first-order corrections to the Maxwell
equations of eleven-dimensional supergravity were found in
\cite{Duff:1995wd} by Duff, Liu and Minasian. They are often called
the Green--Schwarz corrections and are needed by demanding the
cancellation of anomalies in the worldvolume theory of the fivebrane.
As we will review below, the Green--Schwarz term takes the form of a
correction to the Chern--Simons term by adding to $F \wedge F$ an
8-form made out of the first and second Pontryagin forms.

What is still unclear are the corrections to the spinor connection
defined from the supersymmetry variation of the gravitino.  This
connection encodes the geometry of the supersymmetric supergravity
backgrounds: not just they define the notion of Killing spinor, but
the bosonic field equations are equivalent to the vanishing of the
gamma-trace of its curvature \cite{GauPak}.  Due to the incomplete
knowledge of the corrections to this connection, we lack the notion of a
supersymmetric M-theory background (even to first order).

This motivates the search for eleven-dimensional Lorentzian geometries
which admit spinor fields which are parallel relative to connections
which are ``close'' (in a sense which is made precise in Section 4.2) to the
connection in eleven-dimensional supergravity.  In this paper we
explore Lorentzian Sasakian manifolds admitting such spinor fields and obtain
two main results.  

The first is a characterisation of Sasakian manifolds, both in
Riemannian and Lorentzian settings, whose transverse geometry is Ricci
flat as those admitting nonzero ``generalised Killing spinors'', a
notion which is given in Definition \ref{def:generdef}. This
generalises the characterisation of Sasaki--Einstein manifolds as
those admitting Killing spinors. The second result is showing that a
certain class of Lorentzian Sasakian spin manifolds which can be
exhibited as bundles over Calabi--Yau 5-folds admit generalised
Killing spinors and solve the first-order M-theory bosonic equations
with a nonzero flux $F$.

We remark that these bundles are in general nontrivial and that the
Riemannian metric of the Calabi-Yau 5-fold is essentially supported
over the maximally non-integrable distribution which is naturally
associated with the Sasakian manifold. This class of backgrounds is
therefore complementary to previous solutions found on the ``warped
compactifications'' of Calabi-Yau 5-folds by Haupt, Lukas and Stelle
in \cite{MR2511373} (see also \cite{LPSTholcorr}).

This paper is organised as follows. In Section~3 and after some
preliminaries on Sasakian geometry in Section 2, we define the notion
of a generalised Killing spinor and prove Theorem~\ref{thm:GKS}, which
characterises those Sasakian spin manifolds admitting generalised
Killing spinors. They turn out to be a special class of
$\eta$-Einstein Sasakian manifolds with Ricci-flat transverse
geometry. In Section~4 we apply this result to M-theory. We show that
the class of manifolds in Section~3 almost (but not quite) provide
supersymmetric backgrounds of eleven-dimensional supergravity, but
subject to an additional condition on the transverse geometry, they do
provide backgrounds satisfying the corrected M-theory equations and
admitting generalised Killing spinors. This is the content of
Theorem~\ref{thm:lastthm}. We end the paper with some comments on the
existence of the relevant Lorentzian Sasakian spin manifolds.

\subsubsection*{Notation}

Throughout the paper, we consider Clifford algebras as defined, for
instance, in \cite{LM}.  According to this, the Clifford product of
vectors of the standard basis of $\bR^{p,q}$ or $\bC^{p,q}$ is
$e_i\cdot e_j+e_j\cdot e_i=-2 g_{ij}$ and not ``$+2 g_{ij}$''.

\section{Preliminaries on Sasakian manifolds}
\label{sec:preliminaries}

Let $M$ be a smooth manifold of dimension $2n+1$ that is endowed with
a metric $g$ either positive-definite (we set $\epsilon=1$ in this
case) or Lorentzian (namely with signature $(2n,1)$ and we set
$\epsilon=-1$).

\begin{definition}(\cite{Bohle,BaumLeitner})
The pair $(M,g)$ is called a \emph{Sasakian manifold} if there exists
a vector field $\xi$ such that
\begin{enumerate}[label=(\roman*)]
\item $\xi$ is a Killing vector field of constant length
  $\epsilon=g(\xi,\xi)$, and
\item the endomorphism $\Phi=-\nabla\xi:TM\to TM$ satisfies the
  following two conditions for all vectors $X,Y \in TM$:
  \begin{equation*}
    \Phi^2(X) = -X+\epsilon g(\xi,X)\xi
    \end{equation*}
    and
    \begin{equation*}
      (\nabla_X\Phi)(Y) = \epsilon g(X,Y)\xi-\epsilon g(\xi,Y)X~.
  \end{equation*}
\end{enumerate}
\end{definition}

The Killing vector field $\xi$ is called the \emph{characteristic
  vector field} and it gives rise to a subbundle $\cV$ on $M$,
$\cV|_x=\bR\xi|_x$, called the \emph{vertical subbundle}.  The
collection of its $1$-dimensional leaves defines a foliation $\cF_\xi$
on $M$ which we always assume to be regular, i.e., each point of $M$
has a neighbourhood where each leaf passes at most one time.

Every Sasakian manifold comes with the \emph{characteristic $1$-form}
$\eta\in\Lambda^1(M)$ defined by $\eta(X)=\epsilon g(\xi,X)$ for all
$X\in TM$. It is a contact $1$-form and the associated maximally
nonintegrable distribution $\cH=\Ker\eta$ is called the
\emph{horizontal subbundle}.

The tangent bundle of $M$ decomposes into the $g$-orthogonal direct
sum $T M=\cH\oplus\cV$ of the horizontal and vertical subbundles and
$g=g|_{\cH}+\epsilon\eta\otimes\eta$.  We call the associated sections
\emph{horizontal} and, respectively, \emph{vertical} vector fields on
$M$.  Some basic properties of Sasakian manifolds are
\begin{gather*}
  \Phi(\xi)=0~,\qquad\eta(\Phi(X))=0~,\\
  g(\Phi(X),\Phi(Y))=g(X,Y)-\epsilon \eta(X)\eta(Y)~,\\
  d\eta(X,Y)=-2\epsilon g(\Phi(X),Y)~,\\
  \imath_{\xi}d\eta=0~,\qquad \imath_{\xi}\eta=1~.
\end{gather*}
The last equation says that $\xi$ is the Reeb vector field associated
with $\eta$.

One important feature of a (regular) Sasakian manifold is the fact
that the restriction $\epsilon\cdot\Phi\big|_{\cH}$ of $\epsilon\cdot\Phi$ to
$\cH$ is an integrable complex structure which naturally turns the
transverse geometry of the characteristic foliation into a Kähler
manifold.
\begin{theorem}\label{thm:bundle00}
  Let $(M,g,\xi,\Phi,\eta)$ be a $2n + 1$ dimensional compact Sasakian
  manifold. If $B$ is the leaf space of $\cF_\xi$ with the natural
  projection $\pi:M\to B$ then:
  \begin{enumerate}[label=(\roman*)]
  \item $B$ is a compact complex manifold with a Kähler metric $h$ and
    integral Kähler form $\omega$ satisfying $d\eta=-2\pi^*\omega$, and
  \item $\pi:(M,g)\to (B,h)$ is a pseudo-Riemannian submersion with
    totally geodesic fibers all diffeomorphic to $S^1$ and it has the
    natural structure of a principal circle bundle.
\end{enumerate}
\end{theorem}

This result is usually stated in the Riemannian setting (see
\cite{MR1792957}). However it also holds in the Lorentzian case, as a
simple consequence of the fact that Sasakian
structures with $\epsilon=1$ and $\epsilon=-1$ are in a one-to-one
correspondence with each other.

\begin{proposition}
  Let $(M,g,\xi,\Phi,\eta)$ be a Sasakian manifold with $\epsilon=\pm
  1$. Then $(M,\overline g,\xi,\overline\Phi,\eta)$ is a Sasakian
  manifold with $\overline\epsilon=-\epsilon$, where $\overline
  g=g-2\epsilon\eta\otimes\eta$ and $\overline\Phi=-\Phi$.  Moreover
  the Levi-Civita connection of $g$ and $\overline g$ are related as
  $\overline\nabla _X Y=\nabla_X Y+2\eta(Y)\Phi(X)+2\eta(X)\Phi(Y)$.
\end{proposition}

There is an inverse construction to Theorem~\ref{thm:bundle00}, which we
briefly recall.  The \emph{contactification} of a symplectic manifold
$(B,\omega)$ is a contact manifold of one dimension higher, introduced
as the total space $M$ of an appropriate bundle $\pi:M\rightarrow B$
endowed with a contact structure $\cH=\Ker \eta$
(see, e.g., \cite{MR1866631,MR0112160}).

If $\omega=d\alpha$ is exact one can simply consider the trivial
bundle $M=B\times\bR$ and the $1$-form $\eta=dt-2\alpha$, where $t$ is
the coordinate on $\bR$.  If $\omega$ represents an integral
cohomology class, Boothby and Wang first considered a principal circle
bundle $\pi:M\rightarrow B$ of Euler class $[\omega]$ and then a
connection $1$-form $A$ with curvature $dA=-2\pi i \pi^*\omega$; the
$1$-form $\eta=-\frac{i}{\pi} A$ is the required contact form on $M$. For
more details on this construction we refer to \cite{MR2397738}.

If $B$ is in addition Kähler with respect to a complex structure $J$ then
$D J=D \omega=0$ where $D$ is the Levi-Civita connection of the
corresponding Hermitian metric $h(\cdot,\cdot)=\omega(\cdot,J\cdot)$
and the contactification $(M,\eta)$ has a natural Sasakian structure
with horizontal subbundle $\cH=\Ker \eta$ and the following tensor
fields:
\begin{enumerate}[label=(\roman*)]
\item $g=\pi^* h+\epsilon\eta\otimes\eta$,
\item $\xi$ is the fundamental field of the action of $S^1$ on $M$ (of
  period $2$), and
\item $\Phi$ is the $(1,1)$-tensor uniquely determined by
  $\Phi\big|_{\cH}=\epsilon\pi^* J$, $\Phi(\xi)=0$.
\end{enumerate}
Any Sasakian structure obtained in this way is called \emph{strongly regular}.

We conclude this section with a basic result on the curvature tensors
of a strongly regular Sasakian manifold.

We first fix some notation: we denote by $\wh U\in\mathfrak{X}(M)$ the
basic lift of a vector field $U\in\mathfrak{X}(B)$ on the base of
$\pi:M\to B$, this is the unique horizontal vector field satisfying
\begin{equation}
\label{eq:blift}
\pi_*\wh U=U\qquad\text{and}\qquad\cL_{\xi}\wh U=0~.
\end{equation} 
The bracket of two basic lifts is
\begin{equation}
\label{eq:brahor}
[\wh U,\wh V]=\wh{[U,V]}+2\omega(U,V)\xi~.
\end{equation}

\begin{proposition}\label{pr:lemma1}
  Let $(M,g,\xi,\Phi,\eta)$ be a strongly regular Sasakian manifold
  with $\epsilon=\pm 1$ and Kähler transverse geometry
  $(B,h,J,\omega)$. Then
  \begin{enumerate}[label=(\roman*)]
  \item the Levi-Civita connection $\nabla$ of $g$ is the unique
    linear connection which satisfies
    \begin{equation}\label{eq:wtLC}
      \begin{aligned}[m]
        \nabla_{\xi}{\xi}&=0 \\
        \nabla_{\wh U}{\xi}&=\nabla_{\xi}\wh U =-\epsilon \wh{JU}\\
        \nabla_{\wh U}\wh V &= \wh{D_{U}V}+\omega(U,V)\xi~;
      \end{aligned}
    \end{equation}
  \item the curvature tensor $R$ of $g$ is the unique
    $(1,3)$-tensor field which satisfies
    \begin{equation}\label{eq:wtcurv}
      \begin{aligned}[m]
        R(\xi,\wh V)\xi &=-\wh V \\
        R(\xi,\wh V)\wh W &= \epsilon h(V,W)\xi\\
        R(\wh U,\wh V)\xi&=0\\
        R(\wh U,\wh V)\wh W &=
        (R_{h}(U,V)W)\widehat{\phantom{x}}+2\epsilon\omega(U,V)
        \widehat{JW}\\
        &\qquad -\epsilon(\omega(V,W) \wh{JU}-\omega(U,W) \widehat{JV})~, 
      \end{aligned}
    \end{equation}
    where $R_h$ is the curvature tensor of $h$;
  \item the Ricci curvature $\Ric$ of $g$ is the unique symmetric
    $(0,2)$-tensor field which satisfies
    \begin{equation}\label{eq:wtricci}
      \begin{aligned}[m]
        \Ric(\xi,\xi)&=2n\\
        \Ric(\xi,\wh U)&=0\\
        \Ric(\wh U,\wh U)&=\Ric_{h}(U,U)-2\epsilon h(U,U)~,
      \end{aligned}
    \end{equation}
    where $\Ric_h$ is the Ricci curvature of $h$.
  \end{enumerate}
\end{proposition}

Points (i) and (ii) are proved by direct computations which use
\eqref{eq:blift}, \eqref{eq:brahor} and the fact that $(B,h,\omega,J)$
is Kähler. To prove (iii) it is convenient to consider local
\emph{adapted frames}, namely oriented $g$-orthonormal frames on $M$
of the form $(\wh e_i,\xi)$ for some $h$-orthonormal frame
$(e_i)_{i=1}^{2n}$ on $B$ with $e_{i+n}=Je_i$ for all $1\leq i\leq n$.
We omit the details for the sake of brevity.

From now on we will tacitly restrict to strongly regular Sasakian
manifolds, but recall that every regular and compact Sasakian manifold
is automatically strongly regular by Theorem~\ref{thm:bundle00}.

\section{Null Sasakian geometry and generalised Killing spinors}

The main aim of this section is to prove Theorem~\ref{thm:GKS}, a
spinorial characterisation of Sasakian structures with Ricci flat
transverse geometry. We remark that by Proposition~\ref{pr:lemma1} such
structures are never Einstein but rather $\eta$-Einstein (see, e.g.,
\cite{MR2200887}) since they satisfy
$ \Ric=\lambda g+\nu\eta\otimes\eta $ with $\lambda=-2\epsilon$ and
$\nu=2(n+1)$.  In particular, Bär's cone construction \cite{Baer},
relating the existence of real Killing spinors to special holonomy
cones, does not apply in our case.  Instead, the relevant definition
is the following.

\begin{definition}\label{def:generdef}
  Let $(M,g,\xi,\Phi,\eta)$ be a $2n+1$-dimensional Sasakian spin
  manifold with $\epsilon=\pm 1$. A \emph{generalised Killing spinor}
  is a non-zero section $\varphi$ of the associated Dirac spinor
  bundle $p:\bS(M)\to M$ which satisfies:
  \begin{enumerate}[label=(\alph*)]
  \item $\nabla_{X}\varphi =
    \tfrac{1}{2}\epsilon\Phi(X)\cdot\xi\cdot\varphi$ for all
    horizontal vectors $X$;
  \item $\nabla_{\xi}\varphi=-\Phi\cdot\varphi$ where $\Phi$ is
    understood as an element of $\so(TM)$.
  \end{enumerate}
\end{definition}

Note that any generalised Killing spinor is nowhere vanishing since it
is non-zero by definition and parallel with respect to a connection on the
spinor bundle.  Our main result is the following

\begin{theorem}\label{thm:GKS}
  Let $(M,g,\xi,\Phi,\eta)$ be a $2n+1$-dimensional Sasakian manifold
  with $\epsilon=\pm 1$ and transverse Kähler geometry
  $(B,h,J,\omega)$. Then $M$ admits a spin structure if and only if
  $B$ admits a spin structure and
  \begin{enumerate}[label=(\roman*)]
  \item if $M$ has a spin structure carrying a generalised Killing spinor
    then $\Ric_h=0$;
  \item conversely if $B$ has a spin structure carrying a parallel
    spinor then $\Ric_h=0$ and $M$ has a generalised Killing spinor.
  \end{enumerate}
  Moreover if $B$ has full holonomy $\SU(n)$ or it is the
  standard complex torus then there are two parallel spinors such that
  the corresponding generalised Killing spinors $\varphi_\pm$ on $M$
  satisfy $\Phi\cdot\varphi_\pm=\pm\epsilon i\frac{n}{2}\varphi_\pm$.
\end{theorem}

The remaining part of this section is essentially devoted to the proof
of Theorem~\ref{thm:GKS}. We first show that the existence of a
generalised Killing spinor forces the transverse geometry to be
Ricci-flat and then prove (ii) and the last claim.

Let $M$ be a Sasakian spin manifold. Recall that the ``$\Gamma$-trace''
$\Tr_{\Gamma}(R)$ of the curvature tensor is the $1$-form on
$M$ with values in the endomorphisms bundle of $\bS(M)$ defined
by
\begin{equation*}
  \Tr_{\Gamma}(R)(X)\cdot\varphi:=(\epsilon\xi\cdot
  R(X,\xi)+\sum_{i=1}^{2n}\wh e_i \cdot R(X,\wh e_i))\cdot\varphi~,
\end{equation*}
for all $X\in TM$ and $\varphi\in \bS(M)$, where $(\wh e_i,\xi)$
is an adapted frame on $M$. If $r$ is the $(1,1)$-Ricci curvature
tensor of $g$, standard arguments and part (iii) of
Proposition~\ref{pr:lemma1}  yield the general identity
\begin{equation}
  \label{eq:gtrace}
  \Tr_{\Gamma}(R)(X)\cdot\varphi = - \tfrac{1}{2} r(X)\cdot \varphi =
  \begin{cases}
    -n\epsilon\xi\cdot\varphi&\text{ if } X=\xi~;\\
    -\tfrac{1}{2}\wh{r_h(U)}\cdot\varphi+\epsilon\wh U\cdot\varphi &
    \text{ if } X=\wh U~,
  \end{cases}
\end{equation}
where $r_h$ is the $(1,1)$-Ricci curvature tensor of $h$.

The expression of the curvature tensor in part (ii) of
Proposition~\ref{pr:lemma1} yields also
\begin{equation*}
  R(\wh U,\xi)\cdot\varphi =\tfrac{1}{2}\epsilon\xi\cdot\wh U\cdot\varphi~.
\end{equation*}
Now, if $\varphi$ is generalised Killing
\begin{equation*}
  \begin{split}
    \nabla_{\wh U}\nabla_{\wh V}\varphi&=\nabla_{\wh
      U}(\tfrac{1}{2}\wh{JV}\cdot \xi\cdot\varphi)\\
    &=\tfrac{1}{2}\nabla_{\wh U}\wh{JV}\cdot \xi\cdot\varphi+
    \tfrac{1}{2}\wh{JV}\cdot\nabla_{\wh U}\xi\cdot\varphi
    +\tfrac{1}{2}\wh{JV}\cdot \xi\cdot\nabla_{\wh U}\varphi\\
    &=\tfrac{1}{2}\nabla_{\wh U}\wh{JV}\cdot \xi\cdot\varphi+
    \tfrac{1}{2}\wh{JV}\cdot\nabla_{\wh U}\xi\cdot\varphi +
    \tfrac{1}{4}\epsilon\wh{JV}\cdot\wh{JU}\cdot\varphi\\
    &=\tfrac{1}{2}\wh{D_U JV}\cdot\xi\cdot\varphi
    -\tfrac{1}{2}\epsilon h(U,V)\varphi
    -\tfrac{1}{4}\epsilon\wh{JV}\cdot\wh{JU}\cdot\varphi \\
    &=\tfrac{1}{2}\wh{JD_U V}\cdot\xi\cdot\varphi
    -\tfrac{1}{2}\epsilon h(U,V)\varphi
    -\tfrac{1}{4}\epsilon\wh{JV}\cdot\wh{JU}\cdot\varphi
  \end{split}
\end{equation*}
and, using \eqref{eq:brahor} and property (b) of
Definition~\ref{def:generdef},
\begin{equation*}
  \begin{split}
  R(\wh U,\wh V)\cdot\varphi &= 2\omega(U,V) \Phi\cdot\varphi -
  \tfrac{1}{4} \epsilon(\wh{JV}\cdot\wh{JU}-\wh{JU}\cdot\wh{JV})
  \cdot\varphi\\
  &= 2h(JU,V)\Phi\cdot\varphi - \tfrac{1}{2}\epsilon(h(U,V) +
  \wh{JV}\cdot\wh{JU}) \cdot\varphi~.
  \end{split}
\end{equation*}
The value of the $\Gamma$-trace on a generalised Killing spinor
$\varphi$ and an horizontal lift $\wh U$ is therefore equal to
\begin{equation*}
  \begin{split}
    \Tr_{\Gamma}(R)(\wh U)\cdot\varphi &= (\epsilon\xi\cdot R(\wh
    U,\xi)+\sum_{i=1}^{2n}\wh e_i \cdot R(\wh U,\wh e_i))\cdot\varphi\\
    &=-\tfrac{1}{2}\epsilon\wh U\cdot\varphi + 2\sum_{i=1}^{2n}\wh e_i
    \cdot h(JU,e_i)\Phi\cdot\varphi\\
    &\qquad -\tfrac{1}{2}\sum_{i=1}^{2n}\wh e_i
    \cdot \epsilon(h(U,e_i)+\wh{Je_i}\cdot\wh{JU})\cdot\varphi\\
    &=-\tfrac{1}{2}\epsilon\wh U\cdot\varphi+2\wh{JU}\cdot
    \Phi\cdot\varphi - \tfrac{1}{2} \epsilon\wh U\cdot\varphi
    -\sum_{i=1}^{n} \epsilon \wh e_i
    \cdot\wh{Je_i}\cdot\wh{JU}\cdot\varphi\\
    &=-\epsilon\wh U\cdot\varphi + 2\wh{JU}\cdot \Phi\cdot\varphi
    - 2 \Phi \cdot \wh{JU} \cdot \varphi\\ 
    &=-\epsilon\wh U\cdot\varphi + 2\wh{JU}\cdot \Phi\cdot\varphi -
    2\wh{JU}\cdot \Phi\cdot\varphi + 2\epsilon\wh{U}\cdot\varphi\\
    &= \epsilon\wh{U}\cdot\varphi~.
  \end{split}
\end{equation*}
Comparing this with \eqref{eq:gtrace} immediately yields
\begin{equation*}
  \wh{r_h(U)}\cdot\varphi=0
\end{equation*}
for all $U\in\mathfrak{X}(B)$ and $\wh{r_h(U)}$ is a null vector of
the distribution $\cH$.  As $g|_{\cH}$ is positive-definite one gets
$\wh{r_h(U)}=0$ and $\Ric_h=0$.

We now describe the relation between the spin structure of the total
space and the base of the characteristic fibration
$\pi:(M,g)\to (B,h)$. We recall here only the facts that we need and
refer to e.g. \cite[\S 5]{MR2657844} for more details on spin
structures and pseudo-Riemannian submersions.

The orthogonal splitting $TM=\cH\oplus\cV$ of the tangent bundle of
$M$ and the fact that $\cV=\Ker\pi_*$ is trivialised by $\xi$ define
an $\mathrm{SO}(2n)$-reduction $P_\sharp\subset P_g$ of the bundle
$p:P_g\to M$ of oriented $g$-orthonormal frames on $M$, where
\begin{equation*}
  \cH = P_\sharp \times_{\SO(2n)} \bR^{2n} \qquad\text{and}\qquad
  P_\sharp/\SO(2n)\bigr|_{x} \simeq \xi\bigr|_{x}~.
\end{equation*}
If $P_h$ is the bundle of oriented $h$-orthonormal frames on $B$, the
natural map
\begin{equation*}
  d\pi:P_{\sharp}\to P_h~,\qquad d\pi(u)=\pi_*\circ u\bigr|_{\bR^{2n}}~,
\end{equation*} 
is an isomorphism on each fiber and it identifies
$P_{\sharp}\simeq \pi^* P_h$.

We say that a (local) section $\wh s:M\to P_{\sharp}$ is \emph{basic}
if there exists a section $s:B\to P_h$ with $d\pi\circ \wh
s=s\circ\pi$; the basic sections and the sections of $P_h$ are in a
one-to-one correspondence.  Similarly, if $W$ is an
$\mathrm{SO}(2n)$-module, an equivariant map $\wh f:P_{\sharp}\to W$
is \emph{basic} if there exists $f:P_h\to W$ such that $\wh f=f\circ
d\pi$. In this case the sections
\begin{equation}\label{eq:salah}
  \begin{aligned}[m]
    \wh\varphi &= [\wh s,\wh f\circ \wh s]:M\to
    P_\sharp\times_{\SO(2n)}W \qquad\text{and}\\
    \varphi &= [s,f\circ s]:B\to P_h\times_{\SO(2n)}W
  \end{aligned}
\end{equation}
of the associated bundles $P_\sharp\times_{\SO(2n)}W\to M$ and
$P_h\times_{\SO(2n)}W\to B$, respectively, are related by the identity
$d\pi\circ\wh \varphi=\varphi\circ \pi$ and all sections $\wh\varphi$
of $P_\sharp\times_{\SO(2n)}W$ which are basic (i.e., they satisfy
this identity for some $\varphi$) are as in \eqref{eq:salah}.

We also say that two elements $u\in \left.P_\sharp\right|_{x}$ and
$u'\in \left.P_\sharp\right|_{x'}$ are \emph{equivalent}, written
$u\sim u'$, if $\pi(x)=\pi(x')$ and there is a local basic section
$\wh s$ with $\wh s(x)=u$ and $\wh s(x')=u'$. The bundle $P_h$ and its
sections are then naturally identifiable with $P_\sharp/\!\sim$ and,
respectively, the basic sections of $P_\sharp$.

Let now $\Ad:\wt P_g\rightarrow P_g$ be a spin structure on $M$, a
double cover of $P_g$ with structure group $\wt G$ isomorphic to
$\Spin(2n+1)$ if $\epsilon=1$ or to $\Spin(2n,1)$ if $\epsilon=-1$,
inducing the canonical covering morphism on each fiber. We consider
also the principal bundle $\wt P_\sharp:=\Ad^{-1}(P_\sharp)$ on $M$
with fiber $\Spin(2n)=\Ad^{-1}(\SO(2n))$. In complete analogy with the
bundles of linear frames, two elements $u\left.\in \wt P_\sharp\right|_x$ and
$u'\left.\in \wt P_\sharp\right|_{x'}$ are called equivalent if
$\pi(x)=\pi(x')$ and there is a local section
$\wh s:M\to \wt P_\sharp$ with $\wh s(x)=u$, $\wh s(x')=u'$ and which
is basic, in the sense that $\Ad(\wh s):M\to P_\sharp$ is basic. By
\cite[Lemma 5]{MR2657844}
\begin{equation*}
  \Ad:\wt P_h:=\wt P_\sharp/\!\sim~\longrightarrow P_h\simeq
  P_\sharp/\!\sim
\end{equation*} 
is a two-sheeted covering and therefore a spin structure on $B$ (the
proof of the lemma is given just in the Riemannian case but it extends
verbatim to the Lorentzian case too).

Conversely if $B$ is endowed with a spin structure
$\Ad:\wt P_h\to P_h$, the pull-back bundle
$\wt P_\sharp:=\pi^*\wt P_h$ is a double cover of $P_\sharp$ and
enlarging its structure group $\Spin(2n)$ to $\wt G$ yields the spin
structure $\wt P_g=\wt P_\sharp\times_{\Spin(2n)}\wt G$ on $M$.  This
argument shows that $M$ admits a spin structure \emph{if and only if}
$B$ does.

Let $\bS$ be an irreducible module for the complex Clifford
algebra $\mathbb{C}l(2n+1)$ so that the Dirac spinor fields on $M$ are
given by the sections of the bundle
$\bS(M)=\wt P_g\times_{\wt G}\bS\simeq \wt
P_\sharp\times_{\Spin(2n)}\bS$.
As $\bS$ is irreducible also for $\mathbb{C}l(2n)$ (see the
classification of Clifford algebras in e.g. \cite{LM}), the
\emph{basic sections} $\wh\varphi : M\to \wt P_\sharp
\times_{\Spin(2n)} \bS$ of the very same bundle are in a
one-to-one correspondence with the spinor fields on $B$ and form a
natural subclass of the spinors on $M$. Similarly a section of the
bundle of Clifford algebras on $B$ is represented by a basic section
of $\wt P_{\sharp} \times_{\Spin(2n)} \mathbb{C}l(2n+1)\to M$ with
values in $\mathbb{C}l(2n)$.

Our aim is to determine the covariant derivatives $D\varphi$ and
$\nabla\wh\varphi$ of a spinor $\varphi$ on $B$ and the corresponding
(basic) spinor $\wh\varphi$ on $M$. Let
$\wh\vartheta:TP_g\rightarrow\bR^{2n+1}$ be the so-called soldering
form of $P_g$, defined by $\wh\vartheta_u(v)=(v^1,\ldots,v^{2n+1})$
where the $v^i$ are the components of $p_*(v)\in T_{p(u)}M$ with
respect to the frame $u$. The restriction of the soldering form to
$P_\sharp$ decomposes into
\begin{equation*}
  \left.\wh\vartheta\right|_{P_\sharp} =
  (\vartheta_\cV,\vartheta_\cH)~,\qquad
  \vartheta_\cH=(d\pi)^*\vartheta~,
\end{equation*}
where $\vartheta:TP_h\to\bR^{2n}$ is the soldering form of $P_h$. We
call a vector $v\in T_u P_\sharp$ \emph{horizontal}
(resp. \emph{vertical}) if $\vartheta_\cV(v)=0$
(resp. $\vartheta_\cH(v)=0$).

Let also $\wh\omega$ and $\omega$ be the Levi-Civita connection
$1$-forms on $P_g$ and $P_h$, respectively.  One has the decomposition
\begin{equation*}
\left.\wh\omega\right|_{P_\sharp}=\begin{pmatrix} \omega_{\cH} &  A \\
  -\epsilon A^T & 0\end{pmatrix}\quad\text{where}\quad A:TP_\sharp\to \bR^{2n}\quad\text{and}\quad
 \omega_{\cH}:TP_\sharp\to \so(2n)~,
\end{equation*}
and $\omega_{\cH}(v)=(d\pi)^*\omega(v)$ for any horizontal $v\in
TP_\sharp$ (see \cite[Lemma 4]{MR2657844}).

\begin{proposition}
  Let $\varphi$ be a spinor on $B$ and $\wh\varphi$ the corresponding
  (basic) spinor on $M$. Then we have
  \begin{enumerate}[label=(\roman*)]
  \item $\nabla_{\wh U} \wh\varphi =
    (D_{U}\varphi)^{\wh{\phantom{cc}}}\!\!+\tfrac{1}{2}\epsilon\Phi(\wh
    U)\cdot\xi\cdot \wh\varphi$ for all $U\in\mathfrak{X}(B)$;
  \item $\nabla_{\xi}\wh\varphi=-\Phi\cdot\wh\varphi$.
  \end{enumerate}
  If $\varphi$ is parallel then $\Ric_h=0$ and $\wh\varphi$ is a
  generalised Killing spinor.
\end{proposition}

\begin{proof}
  By definitions $\wh\varphi=[\wh
  s,\pi^*\Psi]$ for some basic section $\wh s:M\to \wt
  P_\sharp$ and and a map $\Psi:B\to \mathbb
  S$. By the second part of \cite[Lemma 5]{MR2657844} and
  \cite[Prop.~3]{MR2657844} one has for all basic lifts $\wh U$ on $M$
  \begin{equation*}
    \begin{split}
    \nabla_{\wh U}\wh\varphi=\nabla_{\wh U}[\wh s,\pi^*\Psi] &= [\wh
    s,(\wh s^*\wh\omega)(\wh U)\cdot\pi^*\Psi+\pi^*d\Psi(\wh U)]\\
    &=(D_{U}\varphi)^{\wh{\phantom{cc}}}\!\!-\tfrac{1}{2}\epsilon\wh
    s^* A(\wh U)\cdot\xi\cdot [\wh s,\pi^*\Psi]\\
    &=(D_{U}\varphi)^{\wh{\phantom{cc}}}\!\!+\tfrac{1}{2}\epsilon\Phi(\wh
    U)\cdot\xi\cdot \wh\varphi~.
    \end{split}
  \end{equation*}
  A similar computation for the Reeb vector field yields
  \begin{equation*}
    \begin{split}
      \nabla_\xi\wh\varphi = \nabla_\xi[\wh s,\pi^*\Psi] &= [\wh s,(\wh
      s^*\wh\omega)(\xi)\cdot\pi^*\Psi+\pi^*d\Psi(\xi)]\\
      &=[\wh s,(\wh
      s^*\wh\omega)(\xi)\cdot\pi^*\Psi]\,\!\overset{(\nabla_\xi\xi=0)}=[\wh
      s,(\wh s^*\omega_{\cH})(\xi)\cdot\pi^*\Psi]\\
      &=-\Phi\cdot\wh\varphi~.
    \end{split}
  \end{equation*}
  This proves (i) and (ii). The last two claim follow from a standard
  ``$\Gamma$-trace" computation  and these identities.  The proof is
  completed.
\end{proof}

We showed (i) and (ii) of Theorem~\ref{thm:GKS}. To prove the last claim
we first need to recall the explicit description of the Dirac spin
module.  Let $W=\bR^{2n}$ be the standard Euclidean space with
orthonormal basis $(e_i)_{i=1}^{2n}$ and set $V=W\oplus\bR\xi$ with
$g(\xi,\xi)=\epsilon$.  The complexification $W^{\bC}=W\otimes\bC$ of
$W$ decomposes as a direct sum of isotropic subspaces $W^{\bC} =
W^{10}\oplus W^{01}$, where
\begin{equation*}
  W^{10}=\left<e_i^{10}=\tfrac{1}{2}(e_i-ie_{i+n})\middle | 
    i=1,\dots,n\right> \qquad\text{and}\qquad
  W^{01}=\overline{W^{10}}~.
\end{equation*}
Let $U = W^{10}$ and set $\bS=\Lambda^{\bullet} U^*$.  For any
$v,v'\in U$, $w,w'\in\overline{U}$ and $\varphi\in\bS$, the
identities
\begin{align}
  \label{eq:spin1}
  v\cdot \varphi:=-2\imath_{v}\varphi~,\\
  \label{eq:spin2}
  w\cdot \varphi:=w^{\flat}\wedge \varphi~,
\end{align}
satisfy $v\,v'+v'\,v=w\,w'+w'\,w=0$ and $v\,w+w\,v=-2g(v,w)$ and
therefore give an irreducible representation of the Clifford algebra
$\mathbb{C}l(2n) \simeq\bC(2^n)$. This representation splits in the
direct sum $\bS =\bS^+\oplus\bS^-$ of two
irreducible Weyl spinor modules $\bS^+=\Lambda^{\text{even}}U^*$
and $\bS^-=\Lambda^{\text{odd}}U^*$ for the even part of
$\mathbb{C}l(2n)$; they can also be intrinsically described as the
$\pm 1$-eigenspaces of the volume element
$\vol_{2n}:=(-1)^{\frac{n(n+1)}{2}}i^{n}e_1\cdots e_{2n}$.

To obtain an irreducible representation of $\mathbb{C}l(2n+1)
\simeq
\bC(2^n)\oplus\bC(2^n)$, it is sufficient to complement \eqref{eq:spin1}
and \eqref{eq:spin2} with
\begin{equation}
\label{eq:spin3}
\xi\cdot \varphi:=i^{\frac{\epsilon+1}{2}}\vol_{2n}\cdot \varphi~.
\end{equation}

Now, if one fixes an adapted local frame at $x\in M$, the action of
$\Phi\in\so(V)$ on a spinor $\varphi\in\Lambda^{p}U^*$ is given by
\begin{equation}
  \label{eq:omegaspinor}
  \begin{split}
    \tfrac{1}{2}\epsilon\Phi\cdot
    \varphi&=-i\sum_{i=1}^{n}(e_i^{10}\wedge e_i^{01})\cdot
    \varphi=-\tfrac{i}{4}\sum_{i=1}^{n}[e_i^{10},e_i^{01}]\cdot \varphi\\
    &=\tfrac{i}{2}\sum_{i=1}^{n}\imath_{e_i^{10}}((e_i^{01})^\flat\wedge \varphi)
    -\tfrac{i}{2}\sum_{i=1}^{n}(e_i^{01})^\flat \wedge\imath_{e_i^{10}}\varphi\\
    &=\tfrac{n}{4}i \varphi
    -i\sum_{i=1}^{n}(e_i^{01})^\flat\wedge\imath_{e_i^{10}} \varphi = \tfrac{n}{4}i \varphi
    -\tfrac{p}{2}i \varphi\\
    &= \tfrac{n-2p}{4}i \varphi~.
\end{split}
\end{equation}
If $B$ is simply-connected with holonomy $\SU(n)$ then it is
spin and it is endowed with two parallel spinors $\varphi_\pm$ such
that $\varphi_+\in\Lambda^0 U^*$ and $\varphi_-\in\Lambda^n U^*$ at
any point (see \cite[pag.~61]{Wang}). In the non simply-connected case
or if $B$ is the standard complex torus the same is true, provided an
appropriate spin structure is chosen (see \cite{MR1327109,MR1751897}
for the first case and consider the products of the periodic ``Ramond"
spin structures of the circle $S^1$ in the second case).

These observations and \eqref{eq:omegaspinor} imply at once the last
part of Theorem~\ref{thm:GKS}. We collect here, for later use in
Section 4, the equations that are satisfied by the two generalised
Killing spinors $\varphi_\pm$:
\begin{equation}
  \label{eq:attheend}
  \nabla_{X}\varphi_\pm = \tfrac{1}{2} \epsilon
  \Phi(X)\cdot\xi\cdot\varphi_\pm \quad
  (X\in\cH)~,\qquad\nabla_{\xi}\varphi_\pm=\mp\epsilon
  i\tfrac{n}{2}\varphi_\pm~.
\end{equation}

\section{Applications to supersymmetric theories of gravity}
\label{sec:EEME}

In Section 3 we showed that the integrability conditions for the
existence of a generalized Killing spinor as in Definition
\ref{def:generdef} correspond to the transverse geometry of the
Sasakian manifold being Ricci-flat. In this section we will see that
these conditions are tightly related to a generalized Einstein
equation on Sasakian manifolds and, in particular, we will investigate
the existence of M-theory backgrounds on Lorentzian Sasakian manifolds
which are (possibly nontrivial) bundles over Calabi--Yau 5-folds.

\subsection{A generalised Einstein equation on Sasakian manifolds}
\label{sec:gener-einst-equat}

We recall that a bosonic background of eleven-dimensional supergravity
is given by an eleven-dimensional Lorentzian manifold $(M,g)$ endowed with
a closed four form $F\in\Lambda^4M$ subject to two partial
differential equations (see \cite{CJS}): the \emph{Einstein equation}
\begin{equation}
  \label{eq:EE}
  \Ric(X,Y) = \tfrac{1}{2} g(\imath_X F,\imath_Y F) -
  \tfrac{1}{6}g(X,Y)g(F,F)~,
\end{equation}
and the \emph{Maxwell equation}
\begin{equation}
  \label{eq:ME}
  d\star F = -\tfrac{1}{2}F\wedge F~.
\end{equation}
One is usually interested in \emph{supersymmetric} backgrounds, that
is, backgrounds admitting a spin structure and a non-zero spinor
$\varphi\in\Gamma(\bS(M))$ which is pseudo-Majorana and
satisfies \footnote{The perhaps unusual form of this equation  is due
  to our conventions on Clifford algebras and our metric conventions
  being ``mostly plus''.}
\begin{equation*}
  \nabla_X\varphi + i (\tfrac{1}{6}\imath_X F +
  \tfrac{1}{12}X^\flat\wedge F) \cdot\varphi=0~,
\end{equation*}
for all $X\in TM$.  We recall here that $\varphi$ is
\emph{pseudo-Majorana} if it satisfies the reality condition
$j(\varphi)=\varphi$ where $j:\bS(M)\to \bS(M)$ is an
appropriate antilinear involution on the space of Dirac spinors (see,
e.g., \cite{MajoranaSpinors}).  This map can be conveniently described fixing an
adapted local frame at $x\in M$ and identifying each fiber
$\left.\bS(M)\right|_x$ with our model
\eqref{eq:spin1}--\eqref{eq:spin3} of the Dirac spin module
$\bS=\Lambda^\bullet U^*$. Consider first the ``Hodge star operator"
$\star:\bS\to \bS$ given by
\begin{equation*}
  \star(\Lambda^p U^*)\subset \Lambda^{5-p}
  U^*\qquad\text{and}\qquad\varphi\wedge\star
  \varphi'= g(\varphi,\overline\varphi')\varphi_-~,
\end{equation*}
where $\varphi\to\overline\varphi\in\Lambda^\bullet \overline U^*$ is
the standard conjugation, $\varphi,\varphi'\in\bS$ and
$\varphi_-\in\Lambda^5 U^*$ is normalised so that
$g(\varphi_-,\overline\varphi_-)=1$. One can check that $\star$ is an
antilinear involution which is not $\Spin(10,1)$-equivariant as it
satisfies
\begin{equation*}
  \begin{split}
    \star (v\cdot\varphi)&=2(-1)^{|\varphi|}\overline v\cdot\star\varphi~,\\
    \star (w\cdot\varphi)&=\tfrac{1}{2}(-1)^{|\varphi|+1}\,\overline w\cdot\star\varphi~,\\
    \star (\xi\cdot\varphi)&=-\xi\cdot\star\varphi~,
  \end{split}
\end{equation*}
for all $v\in\ U$ and $w\in\overline U$. However, if one sets
\begin{equation*}
  \left.j\right|_{\Lambda^p U^*} :=
  \tfrac{(-1)^{\frac{p(p-1)}{2}}2^{p}}{\sqrt{32}}\left.\star\right|_{\Lambda^p U^*}~,
\end{equation*}
then $j:\bS\to \bS$ is also antilinear, $j^2=\Id$ and
\begin{equation*}
  \begin{split}
    j (v\cdot\varphi)&=-\overline v\cdot j\varphi~,\\
    j (w\cdot\varphi)&=-\overline w\cdot j\varphi~,\\
    j (\xi\cdot\varphi)&=-\xi\cdot j\varphi~.
  \end{split}
\end{equation*}
It follows that this map is $\Spin(10,1)$-equivariant and it induces
the required pseudo-Majorana conjugation
$j:\bS(M)\to \bS(M)$. We remark for later use that
$j(\Lambda^0 U^*)\subset \Lambda^5 U^*$ at any point $x\in M$.

The equations \eqref{eq:ME} receive higher order corrections in
M-theory. Before turning to them, we first focus on \eqref{eq:EE}, a
modification of the classical Einstein equations which makes sense on
any Sasakian manifold of dimension $2n+1$, $\epsilon=\pm 1$. As usual
we denote the associated characteristic fibration by
\begin{equation*}
  \pi:(M,g,\xi,\Phi,\eta)\to (B,h,J,\omega)~,
\end{equation*}
and our ansatz on the flux is
\begin{equation}
  \label{eq:ansatz}
  F=\lambda\pi^*\omega^2~,
\end{equation}
where $\lambda$ is some real constant. This form is exact, and hence
closed, by (i) of Theorem~\ref{thm:bundle00}.

\begin{theorem}\label{thm:thm2}
A Sasakian manifold is a solution of the Einstein equations with
$F=\lambda\pi^*\omega^2$ if and only if
$\lambda^2=-\frac{6\epsilon}{n-1}$ and $\Ric_h=2\epsilon(n-5)h$.  If
this is the case the metric $g$ is necessarily Lorentzian and the flux
is non-zero.
\end{theorem}

\begin{proof}
  According to the orthogonal decomposition $TM=\cH\oplus\cV$, the
  equation \eqref{eq:EE} splits into three components. The one with
  $X=\xi$ and $Y$ horizontal is automatically satisfied, since
  $\Ric(\xi,\cH)=0$ (see \eqref{eq:wtricci}) and $\imath_{\xi} F=0$.

  Fix now an $h$-orthonormal frame $(e_i)_{i=1}^{2n}$ on $B$ which
  satisfies $e_{i+n}=Je_i$ for all $1\leq i\leq n$ and let
  $(e^i)_{i=1}^{2n}$ be the corresponding $h$-dual frame so that
  \begin{equation*}
    \omega=\sum_{i=1}^n e^i\wedge e^{i+n}~.
  \end{equation*}
  One has
  \begin{equation*}
    \begin{split}
      g(F,F) &= \lambda^2 h(\omega^2,\omega^2)\\
      &= \lambda^2 h(\sum_{i,j=1}^{n}e^i\wedge e^{i+n}\wedge e^j\wedge
      e^{j+n}, \sum_{i,j=1}^{n}e^i\wedge e^{i+n}\wedge e^j\wedge
      e^{j+n})\\
      &= 2n(n-1)\lambda^2
    \end{split}
\end{equation*}
and, by a similar computation, $g(\imath_{\wh U} F,\imath_{\wh U} F)=
4(n-1)\lambda^2 h(U,U)$
for any horizontal lift $\wh U$. This and the last equation of
\eqref{eq:wtricci} yield that the $\cH$-component of the Einstein
equations is satisfied if and only if
\begin{equation}
\label{eq:above}
\Ric_h=(\tfrac{1}{3}(6-n)(n-1)\lambda^2+2\epsilon)\,h~.
\end{equation}

Finally the $\cV$-component holds if and only if 
\begin{equation}
\label{eq:coeff}
\lambda^2=-\frac{6\epsilon}{n-1}
\end{equation} 
as $\Ric(\xi,\xi)=2n$ and
$\tfrac{1}{2}g(\imath_{\xi} F,\imath_{\xi}
F)-\frac{1}{6}g(\xi,\xi)g(F,F)=-\frac{1}{3}\epsilon n(n-1)\lambda^2$.
This gives the last two claims of the theorem and by substituting
\eqref{eq:coeff} in \eqref{eq:above} also $\Ric_h=2\epsilon(n-5)h$.
\end{proof}

We note that the base of the characteristic fibration of a solution of
the Einstein equations is Kähler-Einstein and it is Ricci-flat
precisely in the $11$-dimensional case. It is this fact that will
ultimately allow us to use Theorem~\ref{thm:GKS} in the case $n=5$,
$\epsilon=-1$ and to discuss the existence of a particular kind of
pseudo-Majorana spinors; before doing so we have a look to the
equations \eqref{eq:ME}.

\emph{From now on and in the rest of the paper we restrict to the
  $11$-dimensional Lorentzian case}.

First of all \begin{equation*}
  -\tfrac{1}{2}F\wedge F=-\tfrac{1}{2}\lambda^2\pi^*\omega^4~.
\end{equation*}
To compute the l.h.s. of \eqref{eq:ME}, we need two relations:
\begin{enumerate}[label=(\alph*)]
\item the volume of $g$ is $\vol=-\eta\wedge\pi^*\vol_h$, where
  $\vol_h$ is the volume of $h$, and
\item for any $p$-form $\beta$ on $B$ one has
  $\star(\pi^*\beta)=-(-1)^{p}\eta\wedge\pi^*(\star\beta)$.
\end{enumerate}
To see these fix an adapted frame $(\wh e_i,\xi)$ and the associated
$g$-dual frame $(\wh e\,{}^i,-\eta)$. Then (a) is a consequence of the
identities $\vol_h=e^1\wedge\cdots\wedge e^{2n}$,
$\vol=-\eta\wedge\wh e\,{}^1\wedge\cdots\wedge \wh e\,{}^{10}$ and
$\wh e\,{}^i=\pi^*e^i$.  Now
$\wh\alpha\wedge\star(\pi^*\beta)=g(\wh\alpha,\pi^*\beta)\vol$ for any
$p$-form $\wh\alpha$ on $M$ and
\begin{equation*}
  g(\wh\alpha,\pi^*\beta)=\begin{cases}
    0 & \text{if } \wh\alpha=\eta\wedge\pi^*\alpha \text{ for some } \alpha\in\Lambda^{p-1}B~;\\
    h(\alpha,\beta)& \text{if } \wh\alpha=\pi^*\alpha \text{ for some } \alpha\in\Lambda^{p}B~.
  \end{cases}
\end{equation*}
Point (b) follows then from the case $\wh\alpha=\pi^*\alpha$ and
\begin{equation*}
  \begin{split}
    \wh\alpha\wedge\star(\pi^\star\beta)&=h(\alpha,\beta)\vol=-\eta\wedge h(\alpha,\beta)\pi^*\vol_h\\
    &=-\eta\wedge\wh\alpha\wedge\pi^*(\star\beta)=\wh\alpha\wedge(-(-1)^{p}\eta\wedge\pi^*(\star\beta))~.
  \end{split}
\end{equation*}
Using these two relations, one gets
\begin{equation*}
  \begin{split}
    d\star F&=\lambda d\star(\pi^*\omega^2)=-\lambda
    d(\eta\wedge\pi^*(\star\omega^2))=-\tfrac{1}{3}\lambda
    d(\eta\wedge\pi^*\omega^3)\\
    &=-\tfrac{1}{3}\lambda d\eta\wedge\pi^*\omega^3\\
    &=\tfrac{2}{3}\lambda\pi^*\omega^4~.
\end{split}
\end{equation*}
On the other hand $\lambda^2=\frac{3}{2}$ if the Einstein equations hold,
implying that the Maxwell equations are not satisfied, at least with
the right coefficients. It seems that this issue cannot be fixed in
any straightforward way; we will however shortly see that one has
interesting consequences on M-theory version of the supergravity
equations.

\subsection{M-theory on Calabi-Yau $5$-folds} 

The first and second \emph{Pontryagin forms} of $(M,g)$ are the forms
$p_1\in\Lambda^4M$ and $p_2\in\Lambda^8M$ given by
\begin{equation*}
  p_1=-\frac{1}{8\pi^2}\Tr R^2~,\qquad p_2=\frac{1}{128\pi^4}((\Tr R^2)^2-2\Tr R^4)~,
\end{equation*}
where, for any positive integer $k$, the \emph{trace forms} are
\begin{multline*}
  \Tr R^{2k} (X_1,\ldots,X_{4k})=\\
  \frac{1}{4k!}\sum_{\sigma}\epsilon(\sigma)\Tr(R(X_{\sigma(1)},X_{\sigma(2)})\circ\cdots\circ
  R(X_{\sigma(4k-1)},X_{\sigma(4k)}))~,
\end{multline*}
and the summation is taken over all permutations $\sigma$ of
$\{1,\ldots,4k\}$.  The first-order corrections of the Maxwell
equations that we are interested in are (see \cite{Duff:1995wd} and,
e.g., also \cite{GauPak})
\begin{equation}
  \label{eq:ME2}
  d\star F+\tfrac{1}{2}F\wedge F=-\beta p(M,g)
\end{equation}
where $\beta$ is a real constant  and $p(M,g)$
is the $8$-form on $M$ given by
\begin{equation*}
  \begin{split}
    p(M,g) &= 64\pi^4(p_1^2-4p_2)\\
    &= 4\Tr R^4-(\Tr R^2)^2~.
  \end{split}
\end{equation*}
We remark that $\beta$ is a dimensionful \emph{nonnegative} constant,
proportional to the sixth power of the eleven-dimensional Planck
length; after a unit of length had been fixed once and for all, it is
entirely natural to look for backgrounds which satisfy \eqref{eq:ME2}
for a definite value of this ``parameter''. We will see that this is
indeed the case and that, at least for the class of backgrounds
considered in this paper, this value is automatically dictated.

To state the main Theorem~\ref{thm:lastthm} of this section, we need
some preliminary notions and results. The \emph{Ricci form}
$\rho_1(U_1,U_2)=\Ric_h(JU_1,U_2)$ of a Kähler manifold
$(B,h,J,\omega)$ is related to the trace of its curvature by (see,
e.g., \cite[Vol.II]{KobayashiNomizu}):
\begin{equation*}
  \rho_1(U_1,U_2)=\tfrac{1}{2}\Tr(J\circ R_h(U_1,U_2))\,;
\end{equation*} 
we similarly define  the \emph{second Ricci form} $\rho_2\in\Lambda^6B$ by
\begin{multline*}
  \rho_2(U_1,\ldots,U_{6})=\\
  \tfrac{1}{6!} \sum_{\sigma} \epsilon(\sigma) \Tr(J\circ
  R_h(U_{\sigma(1)},U_{\sigma(2)}) \circ
  R_h(U_{\sigma(3)},U_{\sigma(4)})\circ
  R_h(U_{\sigma(5)},U_{\sigma(6)}))~,
\end{multline*}
where the summation is over all permutations $\sigma$ of
$\{1,\ldots,6\}$.

\begin{definition}
A $10$-dimensional Kähler manifold is called \emph{admissible} if
$\rho_1\in\Lambda^2 B$, 
$\Tr R_h^2\in\Lambda^4 B$, $\rho_2\in\Lambda^6
B$ and $\Tr R_h^4\in\Lambda^8 B$ are all zero.
\end{definition}

We note that any admissible Kähler manifold is in particular Ricci-flat. 
The bridge of Sasakian manifolds with M-theory is provided
by the following.

\begin{proposition}\label{pr:propon}
  Let $M$ be a
  Lorentzian Sasakian manifold with a Kähler base $B$ which is admissible. Then
  $p(M,g)=-\frac{6688}{105}\pi^*\omega^4$.
\end{proposition}

This result is a consequence of the fact that the trace forms of a
general $11$-dimensional Lorentzian Sasakian manifold are
\begin{equation}
  \label{eq:P1}
  \Tr R^2=\pi^*\left\{\Tr R_h^2-\tfrac{4}{3}\rho_1\wedge\omega-8\omega^2\right\}
\end{equation}
and
\begin{multline}
\label{eq:P2}
  \Tr R^4 = \pi^*\left\{\Tr R_h^4-\tfrac{2}{7}\rho_2\wedge\omega
    -\tfrac{2}{35}\Tr R_h^2\wedge\omega^2\right.\\ +
  \tfrac{8}{315}\left.
    \rho_1\wedge\omega^3+\tfrac{8}{105}\omega^4\right\}~.
\end{multline}
These equations are obtained by somewhat long and tedious computations
and a repeated use of the algebraic Bianchi identities.  We omit the
details.

We remark that there does \emph{not} exist any definitive notion of a
\emph{supersymmetric} solution of equations \eqref{eq:EE} and
\eqref{eq:ME2}; more precisely there are no complete results for the
corrections required at order $\beta$ to the covariant derivative
$\nabla^o_X\varphi:=\nabla_X\varphi+i(\frac{1}{6}\imath_X
F+\frac{1}{12}X^\flat\wedge F)\cdot\varphi$.
As a matter of fact the modifications described so far in the
literature had all been obtained by requiring that supersymmetry is
preserved on some particular classes of backgrounds (see
\cite{MR860366}; see also \cite{LPSTholcorr} and references therein
for a more recent discussion on this circle of ideas). In this regard,
we also have to note that two first order corrections of $\nabla^o$
are usually seen as equivalent if they both act trivially on the same
putative parallel spinor $\varphi$.

The modified spinorial connection
$\nabla^\beta$ is not arbitrary but it has to satisfy two basic properties:
\begin{enumerate}[label=(\roman*)]
\item if $\beta=0$ then $\nabla^\beta=\nabla^o$, and
\item $\nabla^\beta$ depends just on $g$ and $F$ and not on any
  geometric datum specific of the backgrounds considered (e.g.
  $\omega$, $\xi$ or $\eta$ in our case).
\end{enumerate}

The following proposal satisfies these two properties.

\begin{definition}
  We call a solution of \eqref{eq:EE} and \eqref{eq:ME2} 
  \emph{supersymmetric} if it has a spin structure and a non-zero
  pseudo-Majorana spinor parallel with respect to the connection
  \begin{equation*}
    \nabla^\beta_X\varphi := \nabla^o_X\varphi +
    i\beta\left\{\mu_1\imath_{X}p(M,g)+\mu_2X^\flat\wedge p(M,g)
    \right\}\cdot\varphi~,
  \end{equation*}
  where 
  \begin{equation*}
    \mu_1=1-\tfrac{1+489\sqrt 6}{1200} \qquad\text{and}\qquad
    \mu_2=\tfrac{1}{12}\cdot\tfrac{\sqrt{3}-\sqrt{2}}{(3\sqrt{3}+4\sqrt{2})}~.
  \end{equation*}
\end{definition}

Our main result is then the following

\begin{theorem}
  \label{thm:lastthm}
  Any Lorentzian Sasakian manifold $M$ with an admissible base $B$ is
  a solution of \eqref{eq:EE} and \eqref{eq:ME2} with the non-zero flux
  \begin{equation*}
    F=\sqrt{\frac{3}{2}}\pi^*\omega^2~.
  \end{equation*}
  If the base $B$ is admissible with full holonomy $\SU(n)$ or it is
  the standard complex torus with its periodic ``Ramond'' spin
  structure then the solution is supersymmetric.
\end{theorem}

\begin{proof}
From $\Ric_h=0$ and Theorem~\ref{thm:thm2} one knows that
\eqref{eq:EE} is satisfied with $\lambda^2=\frac{3}{2}$.  By the
discussion in  \S 4.1 and Proposition~\ref{pr:propon}, equation
\eqref{eq:ME2} holds if and only if $\lambda>0$ and
\begin{equation}
  \label{eq:beta}
  \beta=\tfrac{105}{6688}(\sqrt{\tfrac{2}{3}}+\tfrac{3}{4})~.
\end{equation}
This proves the first part of the theorem, where
$\lambda=\sqrt{\frac{3}{2}}$.

If $B$ is the standard complex torus or has holonomy $\SU(n)$
then Theorem~\ref{thm:GKS} applies and $M$ has two generalised Killing
spinors $\varphi_\pm$ satisfying \eqref{eq:attheend} with $n=5$,
$\epsilon=-1$. We recall that fixing an adapted frame at $x\in M$ as
at the end of \S 1.2 yields appropriate identifications
$\left.TM\right|_{x}\simeq V=W\oplus\bR\xi$ and
$\left.\bS(M)\right|_{x}\simeq \bS=\Lambda^\bullet U^*$ in
such a way that $\varphi_+\in \Lambda^0 U^*$ and
$\varphi_-\in\Lambda^5 U^*$. In particular,
\begin{equation*}
  \begin{split}
    e_i\cdot e_{i+5}\cdot\varphi_\pm &= i(e_i^{10}+e_i^{01})\cdot
    (e_i^{10}-e_i^{01})\cdot\varphi_\pm\\
    &= -i [e_i^{10},e_i^{01}]\cdot\varphi_\pm=\pm i \varphi_{\pm}~,
  \end{split}
\end{equation*}
for any $i=1,\ldots, 5$,
\begin{equation*}
  \pi^*\omega^2\cdot\varphi_\pm = 2 \sum_{1\leq i<j\leq 5}^5 e_i\cdot
  e_{i+5}\cdot e_{j}\cdot e_{j+5}\cdot\varphi=-20\varphi_\pm~,
\end{equation*}
and
\begin{multline*}
  \pi^*\omega^4\cdot\varphi_\pm = \\
  24\!\!\sum_{1\leq i<j<k<l\leq 5}^5 e_i\cdot e_{i+5}\cdot
  e_{j}\cdot e_{j+5}\cdot e_k\cdot e_{k+5}\cdot e_{l}\cdot
  e_{l+5}\cdot\varphi=120\varphi_\pm~.
\end{multline*}

Our first aim is to show $\nabla^\beta_\xi \varphi_\pm=0$, where the
parameter $\beta$ is as in \eqref{eq:beta}.  On the one hand
$\nabla_{\xi}\varphi_\pm=\pm i \frac{5}{2}\varphi_\pm$ and from
\begin{equation*}
  \imath_\xi\pi^*\omega^2\cdot\varphi_\pm=0~,\quad\text{and}\quad
  (\xi^\flat\wedge\pi^*\omega^2)\cdot\varphi_\pm=\xi^\flat\cdot\pi^*\omega^2\cdot\varphi_\pm=\mp
  20\varphi_\pm~,
\end{equation*}
we find that $\nabla^o_\xi\varphi_\pm=\pm 5i(\tfrac{1}{2}-\frac{1}{\sqrt
  6})\varphi_\pm$.  From this fact, Proposition~\ref{pr:propon} and
\begin{equation*}
  \imath_\xi\pi^*\omega^4\cdot\varphi_\pm=0~, \quad\text{and}\quad 
  (\xi^\flat\wedge\pi^*\omega^4)\cdot\varphi_\pm =
  \xi^\flat\cdot\pi^*\omega^4\cdot\varphi_\pm = \pm
  120\varphi_\pm~,
\end{equation*} 
one finally gets $\nabla^\beta_\xi\varphi_\pm=0$.

To prove $\nabla^\beta_X\varphi_\pm=0$ for all horizontal vectors, we
need few additional identities which hold for any spinor
$\varphi\in\Lambda^pU^*$. First a computation similar to
\eqref{eq:omegaspinor} yields
\begin{equation}
  \label{eq:everyspinor}
\pi^*\omega\cdot\varphi=\sum_{i=1}^5 e_i\cdot e_{i+5}\cdot\varphi=(5-2p)i\varphi~.
\end{equation}
Secondly by considering the general relations (see, e.g., \cite{FMPHom})
\begin{equation*}
  (X^\flat\wedge\alpha) \cdot \varphi = X^\flat \cdot \alpha \cdot
  \varphi + \imath_X \alpha\cdot\varphi
\end{equation*}
and
\begin{equation*}
  (X^\flat\wedge\alpha) \cdot \varphi = (-1)^{|\alpha|} \alpha \cdot
  X^\flat \cdot \varphi - \imath_X \alpha \cdot \varphi
\end{equation*}
in the case $\alpha=\pi^*\omega\in\Lambda^2 M$ and
$X\in\mathfrak{X}(M)$ horizontal, one gets
\begin{equation}
  \label{eq:everyspinor2}
X^\flat\cdot\pi^*\omega\cdot\varphi - \pi^*\omega\cdot
X^\flat\cdot\varphi = -2\imath_X\pi^*\omega\cdot\varphi =
2\Phi(X)\cdot\varphi~.
\end{equation}
Using \eqref{eq:everyspinor2} and \eqref{eq:everyspinor} one has for
any horizontal vector
\begin{equation}
  \label{eq:lasteq}
  \begin{split}
    \Phi(X)\cdot\varphi_\pm&=\tfrac{1}{2}(X^\flat\cdot\pi^*\omega-\pi^*\omega\cdot X^\flat)\cdot\varphi_\pm\\
     &=\tfrac{1}{2}i\cdot(\pm 2)X^\flat\cdot\varphi_\pm\\
    &=\pm i X^\flat\cdot\varphi_\pm
  \end{split}
\end{equation}
and since $\varphi_\pm$ is generalised Killing,
\begin{equation*}
  \begin{split}
    \nabla_{X}\varphi_\pm &= -\tfrac{1}{2}\Phi(X)\cdot\xi\cdot \varphi_\pm=\mp\tfrac{1}{2}\Phi(X)\cdot\varphi_\pm\\
    &=-\tfrac{1}{2}i X^\flat\cdot\varphi_\pm~.
  \end{split}
\end{equation*}
A direct computation together with \eqref{eq:lasteq} implies also
\begin{equation*}
  \begin{aligned}[m]
    \imath_X\pi^*\omega^2 \cdot \varphi_\pm &= \mp 8i \Phi(X)\cdot
    \varphi_\pm = 8 X^\flat\cdot\varphi_\pm~,\\
    (X^\flat\wedge\pi^*\omega^2) \cdot \varphi_\pm &=
    X^\flat\cdot\pi^*\omega^2\cdot\varphi_\pm +
    \imath_X\pi^*\omega^2\cdot\varphi_\pm = - 12 X^\flat\cdot\varphi_\pm~,
\end{aligned}
\end{equation*}
and finally $\nabla^o_X\varphi_\pm=i\frac{\sqrt2-\sqrt3}{2\sqrt
  3}X^\flat\cdot\varphi_\pm$.  From this fact, Proposition~\ref{pr:propon} and
\begin{equation*}
  \begin{aligned}[m]
    \imath_X\pi^*\omega^4\cdot\varphi_\pm&=\pm 96 i
    \Phi(X)\cdot\varphi_\pm=-96 X^\flat\cdot\varphi_\pm~,\\
    (X^\flat\wedge\pi^*\omega^4) \cdot \varphi_\pm &=
    X^\flat\cdot\pi^*\omega^4\cdot\varphi_\pm +
    \imath_X\pi^*\omega^4\cdot\varphi_\pm = 24 X^\flat\cdot\varphi_\pm~,
  \end{aligned}
\end{equation*}
one gets $\nabla^\beta_X\varphi_\pm=0$ for all horizontal vectors too. 

We have seen $\nabla^\beta\varphi_\pm=0$. To show that the solution is
supersymmetric one has simply to note that the reality condition
$j(\varphi)=\varphi$ is satisfied for an appropriate linear
combination $\varphi=c_+\varphi_++c_-\varphi_-$ with constant
coefficients (recall the description of the pseudo-Majorana
conjugation given in \S 4.1) and that such combination $\varphi$ is still
$\nabla^\beta$-parallel. The proof is completed.
\end{proof}

We now comment on the class of admissible Kähler manifolds. We note
that it is not empty as it includes all flat Kähler manifolds $B$;
the corresponding Lorentzian Sasakian manifolds given by the total
spaces of the $S^1$-bundles $\pi:M\to B$ provide new non-flat M-theory
backgrounds with a nonzero flux $F$.  For instance in the special case
of the standard complex torus the Kähler form is integral and
therefore $M$ is globally defined and compact: it is, in fact, the
compact quotient of the $11$-dimensional simply connected real
Heisenberg group, see e.g. \cite{MR2416653} for its explicit
description.  By Theorem~\ref{thm:lastthm} this solution is also
supersymmetric.

It is a natural problem to understand whether admissible non-flat
Kähler manifolds do actually exist. Slightly more generally one might
also note that the right hand side of equation \eqref{eq:ME2} is given
by an exact form, due to our ansatz \eqref{eq:ansatz}, and then
consider $10$-dimensional Ricci-flat Kähler manifolds for which
$\Tr R_h^2\in\Lambda^4 B$, $\rho_2\in\Lambda^6 B$ and
$\Tr R_h^4\in\Lambda^8 B$ are all constant multiples of the
appropriate powers of the Kähler form (such manifolds too would
determine solutions of \eqref{eq:EE} and \eqref{eq:ME2}, as it easily
follows from equations \eqref{eq:P1} and \eqref{eq:P2}).

In this regard we stress that $\Tr R_h^2$, $\rho_2$ and $\Tr R_h^4$
are all closed and of type $(p,p)$. It follows from Hodge theory and
Serre duality that the above conditions are satisfied up to exact
terms on compact Calabi-Yau $5$-folds with Hodge numbers
$h^{1,1}=h^{2,2}=1$. By a deep result of \cite{MR2511373} all complete
intersection Calabi-Yau $5$-folds which can be defined in a single
projective space are of this type (see \cite{MR928308} for the
definition and basic properties; see also the list given in
\cite[Table~5]{MR2511373}). To get further insight on these manifolds
seems like an extremely difficult task \cite{MO208537}.

\section*{Acknowledgments}

The first author is supported in part by the grant ST/J000329/1 ``Particle
Theory at the Tait Institute'' from the UK Science and Technology
Facilities Council.  The second author is fully supported by a
Marie-Curie research fellowship of the ``Istituto Nazionale di Alta
Matematica" (Italy).  We would like to thank our respective
funding agencies for their support.


\providecommand{\href}[2]{#2}\begingroup\raggedright\endgroup

\end{document}